\documentclass[12pt]{article}
\usepackage{amssymb}
\usepackage{amsmath}

\newtheorem{theorem}{Theorem}

\newtheorem{corollary}[theorem]{Corollary}

\newtheorem{definition}{Definition}
\newtheorem{example}{Example}

\newtheorem{lemma}{Lemma}

\newtheorem{proposition}{Proposition}

\newenvironment{proof}[1][Proof]{\textbf{#1.} }{\ \rule{0.5em}{0.5em}}

\newcommand{\calP}{{\cal P}}
\newcommand{\ep}{\varepsilon}
\newcommand{\dN}{{{\N}}}
\newcommand{\dR}{{{\bf R}}}
\newcommand{\rmd}{{{\mathrm d}}}
\newcommand{\E}{{{\bf E}}}

\newcommand{\prob}{{{\bf P}}}

\newcommand{\calB}{{\cal B}}

\newcommand{\calF}{\mathcal{F}}

\newcommand{\calX}{\mathcal{X}}
\newcommand{\calY}{\mathcal{Y}}

\renewcommand{\(}{\bigl(}
\renewcommand{\)}{\bigr)\vphantom{)}}

\newcommand{\D}{\mathrm{d}}

\newcommand{\ti}{\widetilde}
\newcommand{\One}{{1\hskip-2.5pt{\rm l}}}

\newcommand{\Om}{\Omega}
\newcommand{\om}{\omega}
\newcommand{\al}{\alpha}
\newcommand{\be}{\beta}
\newcommand{\la}{\lambda}

\newcommand{\N}{\mathbb N}
\newcommand{\R}{\mathbb R}

\newcommand{\F}{\mathcal F}
\newcommand{\B}{\mathcal B}
\newcommand{\Ec}{\mathcal E}

\newcommand{\Ninf}{\N\cup\{\infty\}}

\newcommand{\modO}{{\operatorname{mod}\,0}}
\newcommand{\cP}[2]{\mathbb{P}\mskip1.5mu\(\mskip1.5mu#1\mskip1.5mu
 \big|\mskip1.5mu#2\mskip1.5mu\)}
\newcommand{\cE}[2]{\mathbb{E}\mskip1.5mu\(\mskip1.5mu#1\mskip1.5mu
 \big|\mskip1.5mu#2\mskip1.5mu\)}

\newcommand{\sia}{$\sigma$\nobreakdash-algebra}
\newcommand{\sii}{$\sigma$\nobreakdash-ideal}

\newcommand{\valued}[1]{$#1$\nobreakdash-\hspace{0pt}valued}
\newcommand{\measurable}[1]{$#1$\nobreakdash-\hspace{0pt}measurable}

\newcommand{\Null}[1]{$#1$\nobreakdash-\hspace{0pt}null}

\begin{document}

\title{Random Stopping Times in Stopping Problems and Stopping Games%
\thanks{We thank Yuri Kifer for raising the question that led to this paper
and for his comments and David Gilat for the discussions we had on the subject.
The research of Solan was supported by the Israel Science Foundation, Grant \#212/09,
and by the Google Inter-university center for Electronic Markets and
Auctions.
The research of Vieille was supported by the Fondation HEC.}}
\author{Eilon Solan\thanks{School of Mathematical Sciences,
Tel Aviv University, Tel Aviv 69978, Israel. e-mail:
eilons@post.tau.ac.il.},
Boris Tsirelson\thanks{School of Mathematical Sciences,
Tel Aviv University, Tel Aviv 69978, Israel. e-mail:
tsirel@post.tau.ac.il.}, and
Nicolas Vieille\thanks{Economics and Decision Sciences Department, HEC
Paris, and GREG-HEC. e-mail: vieille@hec.fr.}} \maketitle
\date{}

\begin{abstract}
Three notions of random stopping times exist in the literature.
We introduce two concepts of equivalence of random stopping times,
motivated by optimal stopping problems and stopping games respectively.
We prove that these two concepts coincide
and that the three notions of random stopping times are equivalent.
\end{abstract}


\section{Introduction}

In optimal stopping problems, which have been widely studied in the literature,  a stochastic process is given and the decision maker
has to choose a stopping time so as to maximize the expectation of
the stopped process. Such models were  extended first in Dynkin
(1969) to two-player zero-sum stopping games and later to
multiplayer nonzero-sum stopping games. In these games, there is
a finite set $J$ of players and there are as many $\dR^J$\nobreakdash-valued
(payoff) stochastic processes as nonempty subsets of the player
set. Each player chooses a stopping time, and the game terminates
at the minimum of these stopping times, that is, as soon as at least one
player chooses to stop. The payoff to each player  depends on the
identity of the player(s) who first stopped.

Initially such models were studied from a theoretical perspective, see, e.g., Dynkin (1969), Neveu (1975), Bismuth (1977),  Kifer (1971), Hamad\`ene and Lepeltier (2000), Touzi and Vieille (2000), Rosenberg, Solan, and Vieille (2001), and
Shmaya and Solan (2004).
Recently they were  proven useful in stochastic finance and in the pricing of exotic contingent claims, see, e.g.,
Cvitanic, Karatzas and Soner (1998), Kifer (2000), Chalasani and Jha (2001), Kyprianou (2004),
and Hamad\`ene and Zhang (2010)  (see also McConnell and Schwartz (1986)).

In some of these papers, existence of the value and of optimal stopping times (or Nash equilibria for nonzero-sum games) is established under the assumption that some payoff processes are a.s. less than or equal to other payoff processes. To dispense with such assumptions, notions of relaxed, or random, stopping times have been introduced.
The literature uses three such notions:
randomized stopping times (defined in Chalasani and Jha (2001)),
behavior stopping times (defined in Yasuda (1985)),
and mixed stopping times (defined in Aumann (1964)).

In this paper we define a natural concept of equivalence between
random stopping times (in stopping problems)
and we show that the three notions are equivalent.
Next we define the concept of equivalence between random stopping times in stopping games
and we prove that this concept coincides with the concept of equivalence between random stopping times in stopping problems.

The paper is organized as follows.
In Section \ref{RIRV} we study a filtration-free setup; we consider integer-valued random variables to which
we add an external source of randomness
and we define the concept of detailed distribution.
In Section \ref{model} we define the three types of random stopping times
and the concept of equivalence between random stopping times,
and we state one of the two main results of the paper,
namely,
the equivalence between the three types of random stopping times.
This result is proven in Section \ref{Proof}.
In Section \ref{sec:problems} we relate the concept of equivalence between stopping times
to stopping problems.
Finally in Section \ref{games} we study stopping games,
define the concept of game-equivalence between stopping times,
and prove the second main result of the paper,
namely, that the concepts of equivalence and game-equivalence are one and the same.

\section{Randomizing Integer-Valued Random Variables}\label{RIRV}

Throughout the paper we denote by $I := [0,1]$ the unit interval,
by $\calB$ its Borel $\sigma$-algebra, and by $\lambda$ the Lebesgue
measure over $(I,\calB)$. When $(X,\calX,P_1)$ and $(Y,\calY,P_2)$ are two probability spaces,
the product probability space is denoted by $(X \times Y, \calX \times \calY, P_1 \otimes P_2)$.

Let $(\Om,\F,P)$ be a probability space that is fixed throughout the paper.
Let $ N : \Om\to\N \cup \{\infty\}$
be a random variable.
Treating $ \om \in \Om $ as an \valued{\Om} random variable, the joint
distribution $ \prob_N $ of $\om$ and $N$ is the probability measure
on $ \Om \times (\Ninf) $ defined by
\[
\prob_N \( A \times \{n\} \) = P \{ \om\in A : N(\om)=n \} \, , \ \ \ \forall A \in \F, n \in \Ninf.
\]
The marginal distribution of $\om$, in other words, the projection
of $ \prob_N $ to $ \Om $, is $ P $:
\[
\prob_N \( A \times \N \) = P(A) \, , \ \ \ \forall A \in \F,
\]
the projection of $ \prob_N $ to $ \N \cup \{\infty\}$ is the distribution of $ N
$,
\[
\prob_N \( \Om \times \{n\} \) = P ( N=n ) \, , \ \ \ \forall n \in \N \cup \{\infty\},
\]
and the conditional distribution of $ N $ given $ \om $ degenerates
into the atom at $ N(\om) $.

An additional randomness, external to $ \Om $, may be described by
another probability space $(\ti\Om,\ti\F,\ti P)$ and a
measure-preserving map $ \al : \ti\Om \to \Om $.
Treating $ \om = \al(\ti\om) $ as a random variable on $(\ti\Om,\ti\F,\ti P)$,
for every
random variable $ \ti N : \ti\Om \to \N \cup \{\infty\} $,
the joint distribution $ \prob_{\ti N} $ of $\om$ and $\ti N$ is the
probability measure on $ \Om \times (\Ninf) $ defined by
\[
\prob_{\ti N} \( A \times \{n\} \) = \prob_{\ti N,n}(A) = \ti P \left(\{ \ti\om\in\ti\Om
:
  \al(\ti\om)\in A, \, \ti N(\ti\om)=n \}\right) \, ,\ \ \ \forall A \in \F, n \in \N \cup \{\infty\}.
\]
We call $ \prob_{\ti N} $ the \emph{detailed distribution} of $ \ti N $.
The projection of $ \prob_{\ti N} $ to $ \Om $ is $ P $ and the
projection of $ \prob_{\ti N} $ to $ \N \cup \{\infty\}$ is the
distribution of $ \ti N $.
The conditional distribution of $ \ti N $ given $ \om $
need not be degenerate. We have
\[
P(A) = \prob_{\ti N}(A\times(\Ninf) = \sum_{n \in \Ninf} \prob_{\ti N} \( A \times \{n\} \) =
\sum_{n \in \Ninf} \prob_{\ti N.n}(A) \, , \ \ \ \forall A \in \F,
\]
that is, $ P = \sum_{n \in \Ninf} \prob_{\ti N,n} $. The Radon--Nikodym theorem gives
densities
\[
\rho_n = \frac{ \D \prob_{\ti N,n} }{ \D P } \, ; \quad \prob_{\ti N,n} (A)
= \int_A \rho_n \, \D P \, ; \quad \sum_{n \in \Ninf} \rho_n = 1 \text{ a.s.}
\]
The functions $(\rho_n)_{n \in \Ninf}$ are unique a.s. The conditional
distribution of $ \ti N $ given $ \om $,
\[
\ti P( \ti N=n \mid \om ) = \rho_n (\om) \, ,
\]
is also defined a.s. In the nonrandomized case, $ \rho_n(\om) =
\One_{\{n\}} \(N(\om)\) $ a.s.

%%%%%%%%%%%%%%%%%%%%%%%%%%%%%%%%%%%%%%%%%%%%%%%%%%%%%%%%%%%%%%%
\section{Random Stopping Times}
\label{model}

In this section we present the three concepts of random stopping times existing in the literature,
namely, randomized stopping times, behavior stopping times, and
mixed stopping times.
We then define a notion of equivalence between stopping times.

Let $(\calF_n)_{n \in \dN}$ be a filtration in discrete time defined over $(\Omega,\F,P)$.
We assume w.l.o.g.~that $\calF= \calF_\infty := \sigma(\calF_n, n \in \dN)$.

A \emph{stopping time} is a function $\sigma : \Omega \to \dN \cup \{\infty\}$
that satisfies $\{\omega \in \Omega\colon \sigma(\omega) = n\} \in \calF_n$
for every $n \in \dN \cup \{\infty\}$.
When $\sigma(\omega) = \infty$ stopping does not occur in finite time.
To simplify the writing we also refer to this event as ``stopping occurs at time $\infty$.''

\subsection{Randomized Stopping Times}
\label{sec:randomized}

Chalasani and Jha (2001) defined the following concept of randomized stopping time.

\begin{definition}\label{rst}
A \emph{randomized stopping time} is a nonnegative adapted real-valued process $\rho = (\rho_n)_{n \in \dN \cup \{\infty\}}$
that satisfies $\sum_{n \in \dN \cup \{\infty\}}\rho_n(\omega) =1$ for every $\omega \in \Om$.
\end{definition}

The interpretation of a randomized stopping time is that when $\omega$ is the true state of the world,
the probability that the player stops at time $n \in \dN \cup \{\infty\}$ is $\rho_n(\omega)$.
A randomized stopping time $\rho = (\rho_n)_{n \in \N \cup \{\infty\}}$
can be presented as a randomized integer-valued random variable as follows.
Set $(\ti\Omega,\ti\F,\ti P) = (I \times \Om, \calB \times \F, \lambda \otimes P)$,
$\alpha(r,\omega) = \om$ for every $(r,\omega) \in\ti\Om$, and
\begin{equation}
\label{equ:tau:rho}
\ti N_\rho(r,\omega) = \min\left\{n \in \N \colon \sum_{j=1}^n \rho_j(\omega) \geq r\right\},
\end{equation}
where the minimum of an empty set is $\infty$.
The detailed distribution $P_\rho$ of the randomized stopping time $\rho = (\rho_n)_{n \in \dN \cup \{\infty\}}$ is given by
\[ \prob_\rho(A \times \{n\}) = \ti P_{\ti N_\rho}(A \times \{n\})
= \E_P[\One_A \rho_n], \ \ \ \forall A \in \calF, \forall n \in \dN \cup \{\infty\}. \]

\subsection{Behavior Stopping Times}
\label{sec:behavior}

Yasuda (1985) and Rosenberg, Solan, and Vieille (2001) provided the
following definition of a random stopping time.
We call it \emph{behavior stopping time} because of the analogy of this concept to
behavior strategies in game theory (see, e.g., Maschler, Solan, and Zamir, 2013).

\begin{definition}\label{bst}
A \emph{behavior stopping time} is an adapted $[0,1]$-valued process
$\beta = (\beta_n)_{n \in \dN}$.
\end{definition}

The interpretation of a behavior stopping time is that when the
true state of the world is $\omega$, at time $n \in \dN$ the process stops
with probability $\beta_n(\omega)$, conditional on stopping
occurring after time $n-1$.
With probability $\prod_{n\in
\dN}(1-\beta_n)$ the process never stops.
A behavior stopping time $\beta = (\beta_n)_{n \in \dN}$
can be presented as a randomized integer-valued random variable as follows.
Set $(\ti\Omega,\ti\F,\ti P) = (I^\N \times \Om, \calB^\N \times \F, \lambda^\N \otimes P)$,
$\alpha((r_n)_{n \in \N},\omega) = \om$ for every $((r_n)_{n \in \N},\omega) \in\ti\Om$, and
\begin{equation}
\label{equ:tau:beta}
\ti N_\beta((r_n)_{n \in \N},\omega) = \min\left\{n \in \N \colon r_n \leq \beta_n(\omega)\right\}.
\end{equation}
The detailed distribution $\prob_\beta$ of a behavior stopping time $\beta = (\beta_n)_{n \in \dN}$ is
\[ \prob_\beta(A \times \{n\}) = \ti P_{\ti N_\beta}(A \times \{n\}) =
\left\{
\begin{array}{lll}
\E_P\left[ \One_A \left(\prod_{j < n} (1-\beta_j)\right) \beta_n\right], &
\ \ \ & \forall A \in \calF, \forall n \in \dN,\\
\E_P\left[ \One_A \left(\prod_{j \in \dN} (1-\beta_j)\right) \right], &
& \forall A \in \calF, n=\infty.
\end{array}
\right.
\]

\subsection{Mixed Stopping Times}
\label{sec:mixed}

Following Aumann (1964) we define the concept of a mixed stopping
time as follows (see also Touzi and Vieille (2002) and Laraki and Solan (2005)
for an analog concept in continuous-time problems).

\begin{definition}\label{mst}
A \emph{mixed stopping time} is a \measurable{(\calB \times \calF)}
function $\mu : I \times \Omega \to \dN \cup \{\infty\}$ such that for
every $r \in I$, the function $\mu(r,\cdot)$ is a stopping time.
\end{definition}

The interpretation of a mixed stopping time is that $r$ is chosen according to the uniform distribution at the outset,
and the stopping time $\mu(r,\cdot)$ is used.
Aumann's formulation allows us to define a random choice of a stopping time
without imposing the structure of a probability space over the set of stopping times.

A mixed stopping time $\mu$
can be presented as a randomized integer-valued random variable as follows.
Set $(\ti\Omega,\ti\F,\ti P) = (I \times \Om, \calB \times \F, \lambda \otimes P)$,
$\alpha(r,\omega) = \om$ for every $(r,\omega) \in\ti\Om$, and
\begin{equation}
\label{equ:tau:mu}
\ti N_\mu(r,\omega) = \mu(r,\omega).
\end{equation}
The detailed distribution $\prob_\mu$ of a mixed stopping time $\mu$ is
\begin{eqnarray}
\nonumber
\prob_\mu(A \times \{n\}) &=&  \ti P_{\ti N_\mu}(A \times \{n\})\\
\nonumber
&=& (\la\times P) \left(\{ (r,\om) : \om \in A, \, \ti N_\mu(r,\om)=n
\}\right)\\
\label{equ101}
&=&
\int_A P(\D\om) \int_0^1 \D r \,\One_{\{n\}} \( \mu(r,\om) \)\\
&=&
\int_0^1 P(\{\omega \in A \colon \mu(r,\omega) = n\}) \rmd r,
\ \ \ \forall A \in \calF, \forall n \in \dN \cup \{\infty\}.
\nonumber
\end{eqnarray}

\subsection{Equivalence between Random Stopping Times}

Below we will use the symbol $\eta$ to refer to a random stopping time that
can be either randomized, mixed, or behavior.

%Two random stopping times are equivalent if they have the same
%detailed distribution.

\begin{definition}
Two random stopping times $\eta$ and $\eta'$ are \emph{equivalent} if they have the same detailed distribution:
$\prob_{\eta} = \prob_{\eta'}$.
\end{definition}
This definition is the analog to stopping problems of the definition of equivalent strategies in extensive-form games
(see, e.g., Kuhn (1957) or Maschler, Solan, and Zamir (2013)).

We now define the concept of a stopping measure, which will play an important role in the equivalence between
the various types of stopping times.
\begin{definition}\label{sm}
A \emph{stopping measure} is a probability measure $ \nu $ on $ \Om
\times ( \Ninf ) $ whose projection to $\Om$ is $P$,
such that the corresponding densities $ \rho_n $ defined for $
n \in \Ninf $ by $ \nu \( A \times \{n\} \) = \int_A
\rho_n \, \D P $ for every $A \in \F$, satisfy the condition
\[
\rho_n \text{ is equal a.s.\ to some \measurable{\F_n} function}
\]
for all $ n \in \N $ (and therefore also for $n=\infty$).
\end{definition}

Our first main result is an equivalence theorem between the concept of
stopping measures and the three types of random stopping times.
This result implies in particular that a random stopping time of each of the
three types (randomized, behavior, or mixed) has an equivalent
stopping time of each of the other types.

\begin{theorem}
\label{theorem:1}
The following four conditions on a probability measure $ \nu $ on $
\Om \times (\Ninf) $ are equivalent:

\begin{enumerate}
\item[(a)] $\nu$ is a stopping measure.
\item[(b)] $\nu$ is the detailed distribution of some (at least one)
randomized stopping time;
\item[(c)] $\nu$ is the detailed distribution of some (at least one) behavior
stopping time;
\item[(d)] $\nu$ is the detailed distribution of some (at least one) mixed
stopping time.
\end{enumerate}
\end{theorem}

%%%%%%%%%%%%%%%%%%%%%%%%%%%%%%%%%%%%%%%%%%%%%%%%%%%%%%%%%%%%%%%%%%%

\section{Proof of Theorem \ref{theorem:1}}\label{Proof}

For the proof we adopt the ``$\modO$'' approach.
As is well known, a random variable is defined either as a
measurable function on $\Om$ or an equivalence class of such
functions, depending on the context; here ``equivalence'' means ``equality almost surely''. For
example, the conditional expectation $ \cE X \F $ is generally an
equivalence class. Likewise, an event is defined either as a
measurable subset of $\Om$ or an equivalence class of such sets. In
many cases it is possible and convenient to work in terms of
equivalence classes only.
This approach is known as the ``$\modO$'' approach.

The $\modO$ approach to \sia s (assumed to contain only
\measurable{\F} sets) replaces each set of a given \sia\ with the
corresponding equivalence class. Two \sia s are treated as
equivalent if they lead to the same set of equivalence classes. This
holds if and only if they lead to the same operator of conditional
expectation. The largest \sia\ $\overline\Ec$ equivalent to $\Ec$ is
generated by $ \Ec $ and all $P$-null sets.
The $\sigma$-algebra $\overline\Ec$ is called the completion
of $ \Ec $. The notion ``equivalence class of \sia s'' is of little
use; instead the notion of the completion is used.
Every \measurable{\overline\Ec} function is equivalent to some (generally
nonunique) \measurable{\Ec} function.

We now turn to the proof of Theorem \ref{theorem:1}.
When $\rho = (\rho_n)_{n \in \Ninf}$ is a randomized stopping time,
the density of $\prob_{\ti N_\rho,n}$ is $\rho_n$, for every $n \in \Ninf$.
It follows that Condition (b) implies Condition (a).
We now prove the converse implication.

\begin{lemma}\label{lemmaRST2}
Every stopping measure is the detailed distribution of some (at least
one) randomized stopping time.
\end{lemma}

\begin{proof}
Let $\nu$ be a stopping measure, and let $(\rho_n)_{n \in \Ninf}$ be the corresponding densities.
By the definition of a stopping measure there exists an adapted process $\widehat\rho = ({\widehat\rho}_n)_{n \in \Ninf}$
such that ${\widehat\rho}_n=\rho_n$ a.s.~for every $n \in \Ninf$ and
$\sum_{n \in \Ninf} {\widehat\rho}_n = 1$ a.s.
To ensure that $\widehat\rho$ is a randomized stopping time we need
to modify $(\widehat\rho_n)_{n \in \Ninf}$ on a null set, preserving adaptedness, in such a way that
$(\widehat\rho_n)_{n \in \Ninf}$ be nonnegative and satisfy
$\sum_{n \in \Ninf} \widehat\rho_n = 1$ everywhere.

For each $n \in \dN$ define
\[ \widehat{\widehat\rho}_n :=
\max\left\{0,\min\left\{\widehat\rho_n, 1-\sum_{k=1}^{n-1}\widehat{\widehat\rho}_k\right\}\right\}. \]
Then $\widehat{\widehat\rho}_n$ is nonnegative, $\F_n$-measurable,
 satisfies
$\widehat{\widehat\rho}_n = \widehat\rho_n$ a.s.,
and $\sum_{k=1}^n \widehat{\widehat\rho}_k \leq 1$ for every $n \in N$.
Finally set
\[ \widehat{\widehat\rho}_\infty := 1 - \sum_{k=1}^\infty\widehat{\widehat\rho}_k. \]
It follows that $\rho_n = {\widehat{\widehat \rho}}_n$ a.s.~for every $n \in \Ninf$,
and $\nu = \prob_{\widehat{\widehat \rho}}$.
This concludes the proof.
\end{proof}

\bigskip

Given a behavior stopping time $\be$,
the density of $\prob_{\ti N_\beta,n}$ is $\left(\prod_{j < n} (1-\beta_j)\right) \beta_n$ for $n \in \N$,
and $\prod_{j \in \dN} (1-\beta_j)$ for $n=\infty$.
It follows that Condition (c) implies Condition (a).
We now prove the converse implication.
\begin{lemma}
Every stopping measure is the detailed distribution of some (at least
one) behavior stopping time.
\end{lemma}

\begin{proof}
The result holds since a behavior stopping time
$\beta = (\beta_n)_{n \in \N}$ is basically the same as a randomized stopping time $\rho = (\rho_n)_{n \in \Ninf}$;
they differ only in the choice of parameters describing probability
measures on $ \Ninf $: either $\rho_n(\om)=\cP{\ti N=n}\om$ or  $ \beta_n(\om) =
\cP{ \ti N=n }{ \ti N \ge n; \, \om } $ (a discrete-time hazard rate).

Formally, given a stopping measure $\nu$, Lemma~\ref{lemmaRST2} provides a randomized
stopping time $\rho$ such that $ \prob_\rho = \nu $.
Define
\[
\be_n :=
\frac{ \rho_n }{ 1 - \rho_1 - \dots - \rho_{n-1} } \, \ \ \ \forall n \in N,
\]
where, by convention, $\frac00 = 0 $ (any choice of convention will do).
Since the process $(\rho_n)_{n \in \Ninf}$ is adapted, so is the process $(\be_n)_{n \in \Ninf}$.
It is immediate to verify that the detailed distributions of $\rho$ and $\beta$ coincide.
\end{proof}

\bigskip

We next prove that Condition (a) implies Condition (d).

\begin{lemma}\label{lemmaRST4}
Every stopping measure is the detailed distribution of some (at least
one) mixed stopping time.
\end{lemma}

\begin{proof}
Given a stopping measure $\nu$, Lemma~\ref{lemmaRST2} provides a randomized
stopping time $\rho$ such that $ \prob_\rho = \nu $. We construct
$\mu$ as follows:
\[
\mu(r,\om) =
\min\left\{ n \in \N \colon r \leq \sum_{j=1}^n \rho_j(\omega) \right\}. \]
Note that the set $ \{ \om : \mu(r,\om)=n \} =
\left\{ \om : \sum_{j=1}^n \rho_j(\omega) \ge r \right\} \cap \left\{ \om :
\sum_{j=1}^{n-1} \rho_j(\omega) < r \right\} $ belongs to $ \F_n $ since $
\rho $ is adapted.
It follows that $\mu$ is a mixed stopping time.
By (\ref{equ101}),
\begin{eqnarray}
\nonumber
\prob_\mu(A \times \{n\}) &=&
\int_A P(\D\om) \int_0^1 \D r \,\One_{\{n\}} \( \mu(r,\om) \)\\
&=& \int_\Omega \One_A \rho_n \D P,
\ \ \ \forall A \in \calF, \forall n \in \dN \cup \{\infty\}
\end{eqnarray}
so that $\prob_\mu = \prob_\rho = \nu$, as desired.
\end{proof}

\bigskip

It remains to prove that Condition (d) implies Condition (a).
To this end we study a more general question
that has its own interest.
Consider the product of two probability spaces
\[
(\Om,\F,P) = (\Om_1 \times \Om_2, \F_1 \times \F_2, P_1 \otimes P_2) \, ,
\]
a function $ f : \Om \to \R $, and its sections $ f_{\om_1} : \Om_2 \to
\R $ defined by $ f_{\om_1}(\om_2) = f(\om_1,\om_2) $ and $ f^{\om_2} : \Om_1 \to
\R $ defined by $ f^{\om_2}(\om_1) = f(\om_1,\om_2) $. If $ f $ is
\measurable{\F} then all its sections are measurable (w.r.t.\ $\F_2$
and $\F_1$ respectively). That is, joint measurability implies
separate measurability. The converse is generally wrong: separate
measurability does not imply joint measurability.

Assume that a sub-\sia\ $ \Ec_2 \subset \F_2 $ is given, $f$ is
\measurable{\F}, and $f_{\om_1}$ is \measurable{\Ec_2} for all $ \om_1
\in \Om_1 $. Does it follow that $f$ is
\measurable{{(\F_1\times\Ec_2)}}?

As we show below, in this case $f$ is
\measurable{(\overline{\F_1\times\Ec_2})}; this is,
it is measurable w.r.t.~the \sia\ generated
by $\F_1\times\Ec_2$ and all \Null{P} sets.
However, $f$ is not necessarily
\measurable{(\F_1\times\Ec_2)}. Consequently, assuming integrability of
$f$, the averaged function $ g : \Om_2 \to \R $ defined by
\[
g(\om_2) := \int f^{\om_2} \, \D P_1 = \int f(\om_1,\om_2) \,
P(\D\om_1) \,
\]
is \measurable{\overline\Ec_2} but not necessarily \measurable{\Ec_2}.

\begin{proposition}
\label{prop1}
Under the above notations the function $f$ is \measurable{(\overline{\F_1\times\Ec_2})}.
\end{proposition}

\begin{sloppypar}
The proof of Proposition \ref{prop1} is functional\nobreakdash-analytic and is based on the well-known relation
\[
L_2(\Om,\F,P) = L_2 \( (\Om_1,\F_1,P_1), \, L_2(\Om_2,\F_2,P_2) \) \,
.
\]
It means that, first, for every $ f \in L_2(\Om,\F,P) $ its sections
$f_{\om_1}$ belong to $ L_2(\Om_2,\F_2,P_2) $, the vector-function $
\om_1 \mapsto f_{\om_1} $ is \measurable{\bar\F_1}, and $ \int \|
f_{\om_1} \|_2^2 \, P_1(\D\om_1) < \infty $. Second, this map from $
L_2(\Om,\F,P) $ to the space $ L_2 \( (\Om_1,\F_1,P_1), \,
L_2(\Om_2,\F_2,P_2) \) $ of vector-functions is a bijection.
In fact, it is a
linear isometry onto (recall that $L_2$ consists of equivalence
classes rather than functions; everything is treated $\modO$ here).
\end{sloppypar}

\bigskip

\begin{proof}[Proof of Proposition \ref{prop1}]
Fix an \measurable{\F} function $ f $ with
\measurable{\Ec_2} sections.
Assume w.l.o.g.~that $f$ is bounded;
otherwise turn to, say, $ \arctan(f(\cdot)) $.
It follows that $f$
belongs to $ L_2(\Om,\F,P) $. The corresponding vector-function $
\om_1 \mapsto f_{\om_1} $ maps $ \Om_1 $ to $ L_2(\Om_2,\Ec_2,P_2) $.
It thus belongs to $ L_2 \( (\Om_1,\F_1,P_1), \, L_2(\Om_2,\Ec_2,P_2)
\) $, and therefore $ f \in L_2 ( \Om, \F_1\times\Ec_2, P ) $.
\end{proof}

\bigskip

Proposition \ref{prop1} implies the following result.

\begin{corollary}
\label{cor1}
Under the above notations the function $g$ is \measurable{\overline\Ec_2}
\end{corollary}

We provide a second, direct probabilistic proof to Corollary \ref{cor1}.

\bigskip

\begin{proof}
The proof is based on the strong law of large
numbers. Given an integrable function $ f $ with \measurable{\Ec_2} sections we
introduce the product $ (\Om_1^\N,\F_1^\N,P_1^\N) $ of an infinite
sequence of copies of $(\Om_1,\F_1,P_1)$. For $P_2$-almost every $ \om_2 $
the section $ f^{\om_2} $ is integrable. Applying the strong law of
large numbers we have
\begin{equation}\label{*}
\frac1n \sum_{k=1}^n f_{\om_{1,n}} (\om_2) = \frac1n \sum_{k=1}^n
f^{\om_2} (\om_{1,n}) \to \int f^{\om_2} \, \D P_1 = g(\om_2) \quad
\text{as } n\to\infty
\end{equation}
for almost all sequences $(\om_{1,n})_{n\in\N} \in \Om_1^\N $ (the null
set of exceptional sequences may depend on $ \om_2 $). By Fubini
Theorem, there is a sequence such that \eqref{*} holds for almost all
$ \om_2 $.
In fact, almost every sequence will do. Thus, \eqref{*}
represents $ g $ as the almost sure limit of a sequence of
\measurable{\Ec_2} functions, and therefore $ g $ is
\measurable{\overline\Ec_2}.
\end{proof}

\bigskip

\begin{lemma}\label{lemmaRST5}
The detailed distribution of any mixed stopping time is a stopping measure.
\end{lemma}

\begin{proof}
Let $\mu$ be a mixed stopping time.
By \eqref{equ101},
\begin{equation}
\prob_{\ti N_\mu,n}(A) =
 \int_A P(\D\om) \int_0^1 \D r \,\One_{\{n\}} \( \mu(r,\om) \),
\end{equation}
so that the density of $\prob_{\ti N_\mu,n}$ is
\begin{equation}
\label{equ001}
\rho_n(\om) = \int_0^1 \One_{\{n\}} \( \mu(r,\om) \) \, \D r.
\end{equation}
To prove that $\prob_{\ti N_\mu}$ is a stopping measure it is left to prove that $\rho_n$ is \measurable{\overline\F_n}
for every $n \in \Ninf$.
However, this follows from Corollary \ref{cor1} with
$(\Omega_1,\F_1,P_1) = (I,\B,\lambda)$,
$(\Omega_2,\F_2,P_2) = (\Om,\F,P)$,
$ f(\cdot,\cdot) = \One_{\{n\}} \( \mu(\cdot,\cdot) \) $,
and
$\Ec_2=\F_n$.
\end{proof}

\bigskip

We now provide an example that in the setup of Proposition \ref{prop1}
the function $ f $ need not be \measurable{(\F_1\times\Ec_2)}.
This example shows in particular that the densities $(\rho_n)_{n \in \Ninf}$ in Eq.~(\ref{equ001})
need not be $\F_n$-measurable.

The example is based on the so-called meagre sets on $[0,1]$. A
\emph{meagre} set is the union of countably many nowhere dense sets;
and a nowhere dense set is a set whose closure has no interior
points. By the Baire Category Theorem, a meagre set cannot contain an
interval. The meagre sets form a \sii; that is, a subset of a meagre set is
meagre, and the union of countably many meagre sets is meagre. A set
is \emph{comeagre} if its complement is a meagre set. All meagre and
comeagre sets are a \sia.

A closed subset of $[0,1]$ of nonzero Lebesgue measure need not have
interior points, that is, can be meagre; the well-known
Smith-Volterra-Cantor set, called also \emph{fat Cantor set}, is an example.

\begin{example}
Assume that $ \Om_1 = \Om_2 = [0,1] $; $ \F_1 = \F_2 $
is the Borel \sia\ on $[0,1]$; $ P_1 = P_2 $ is Lebesgue measure; $
\Ec_2 $ consists of all meagre and comeagre Borel sets; and
\[
f(\om_1,\om_2) = \begin{cases}
 1 &\text{if } \om_2 \ne 0 \text{ and } \om_1/\om_2 \in K,\\
 0 &\text{otherwise,}
\end{cases}
\]
where $K$ is the fat Cantor set.

The section $ f_{\om_1} $ is the indicator of the set $ \{ \om_2 :
\frac{\om_1}{\om_2} \in K \} = \{ \frac{\om_1}{k} : k \in K, \,
k\ge\om_1 \} $. Being a homeomorphic image of $ K \cap [\om_1,1] $,
this set is meagre, and therefore $ f_{\om_1} $ is \measurable{\Ec_2}. On
the other hand, $ f^{\om_2} $ is the indicator of $ \om_2 K = \{ \om_2
k : k \in K \} $, thus $ g(\om_2) = \la(\om_2 K) = \om_2 \la(K) $.
The function $ g $ fails to be \measurable{\Ec_2}.
Indeed, if $g$ were
\measurable{\Ec_2}, then the set $ \{ \om_2 : g(\om_2) \le 0.5
\la(K) \} $ would belong to $ \Ec_2 $. However, this set is the
interval $ [0,0.5] $, which is neither meagre nor comeagre on $ [0,1] $.
In particular, $f $ fails to be \measurable{(\F_1\times\Ec_2)}.

It is worth noting that $\Ec_2$ does not contain intervals, but
$\overline\Ec_2$ is the whole Lebesgue \sia\ (which follows easily from
existence of a meagre set of full measure).
\end{example}

%%%%%%%%%%%%%%%%%%%%%%%%%%%%%%%%%%%%%%%%%%%%%%%%%%%%%%%%%%%%%%%%%%%

\section{Stopping Problems and the Concept of Equivalence}
\label{sec:problems}

In this section we present the model of stopping problems and explore some implications of the equivalence between random stopping times.

A  real-valued process $X = (X_n)_{n \in \dN \cup \{\infty\}}$ is \emph{integrable}
if $\E_P\left[\sup_{n \in \dN \cup \{\infty\}} |X_n|\right] < \infty$.
A \emph{stopping problem}
is an  adapted integrable real-valued process $X = (X_n)_{n \in \dN \cup \{\infty\}}$.

Fix a stopping problem $X = (X_n)_{n \in \dN \cup \{\infty\}}$.
For every stopping time $\sigma$, the expected payoff induced by $\sigma$ is the expectation
\[ \gamma(\sigma;X) := \E_P\left[X_{\sigma(\omega)}(\omega)\right]\]
of the stopped process.
The expected payoff given a randomized stopping time $\rho$ is
\[ \gamma(\rho;X) := \E_{P}\left[\sum_{n \in \dN \cup \{\infty\}} \rho_{n}X_n\right]. \]
The expected payoff given a behavior stopping time $\beta$ is
\[ \gamma(\beta;X) :=
\E_{P}\left[\sum_{n \in \dN}
\left(\prod_{j < n}(1-\beta_j)\right) \beta_n X_{n}
+ \left(\prod_{j \in \dN}(1-\beta_j)\right) X_\infty\right]. \]

The expected payoff given a mixed stopping time $\mu$ is
\[ \gamma(\mu;X) := \int_{0}^1 \gamma(\sigma(r,\cdot);X)\rmd r. \]

The following theorem shows the significance of the concept of equivalent stopping times:
two stopping times are equivalent if, and only if, they yield the same expected payoff in \emph{all}
stopping problems.

\begin{theorem}
\label{th:2}
Two random stopping times $\eta$ and $\eta'$ are equivalent
if, and only if, $\gamma(\eta;X) = \gamma(\eta';X)$ for every stopping problem $X$.
\end{theorem}

\begin{proof}
 We first rewrite the payoff induced by a random stopping time $\eta$ in a more convenient form.
Let an integrable process $X = (X_n)_{n \in \dN \cup \{\infty\}}$ be given,
and let $\rho$, $\beta$, and $\mu$ be a generic randomized, behavior, and mixed stopping time, respectively. One has
\begin{eqnarray}
\gamma(\rho;X) = \E_P\left[ \sum_{n \in \dN \cup \{\infty\}} \rho_n X_n\right]
= \sum_{n \in \dN \cup \{\infty\}}\E_P\left[\rho_n X_n\right],
\end{eqnarray}
\begin{eqnarray}
\gamma(\beta;X) &=&
\E_P\left[\sum_{n \in \dN} \left(\prod_{j<n}(1-\beta_j)\right) \beta_n X_n
+ \left(\prod_{j \in \dN}(1-\beta_j)\right) X_\infty \right]\\
&=&\sum_{n \in \dN} \E_P\left[\left(\prod_{j<n}(1-\beta_j)\right) \beta_n X_n\right]
+ \E_P\left[\left(\prod_{j \in \dN}(1-\beta_j)\right) X_\infty \right],
\noindent
\end{eqnarray}
and
\begin{eqnarray}
\gamma(\mu;X) &=&
\int_{0}^1 \gamma(\mu(r,\cdot);X)\rmd r\\
&=& \int_{0}^1 \E_P[X_{\mu(r,\cdot)}]\rmd r\\
&=& \sum_{n \in \dN \cup \{\infty\}} \int_{0}^1 \E_P[\One_{\{\mu(r,\cdot)=n\}} X_n]\rmd r.
\end{eqnarray}
\noindent
Let $A \in \calF$ and   $n \in \dN \cup \{\infty\}$ be arbitrary  and define
$X = (X_n)_{n \in \dN \cup \{\infty\}}$  by
\[
X_j = \left\{
\begin{array}{lll}
\prob(A\mid \calF_n), & \ \ \ \ \ & j=n,\\
0, & & \textrm{otherwise}.
\end{array}
\right.
\]
For such $X$, one then has, for $\eta=\rho,\beta,\mu$,
\[ \prob_\eta(A \times \{n\}) = \gamma(\eta;X). \]
This proves the reciprocal implication.

We now turn to the direct implication. and let $\eta$ and $\eta'$ be two equivalent random stopping times.
Consider an arbitrary stopping problem $X$. Given $\ep>0$, let $(S_n)_{n \in \Ninf}$
be an adapted process such that the range of $S_n$ is finite for each $n \in \Ninf$, and
$\E[\sup_n|S_n-X_n|]\leq \ep$. Using the equivalence of $\eta$ and $\eta'$ and the above payoff formulas
one has $\gamma(\eta;S)=\gamma(\eta';S)$. The equality  $\gamma(\eta;X)=\gamma(\eta';X)$ then follows by letting $\ep\to 0$ and by dominated convergence.
\end{proof}

\bigskip

Let $X$ be a stopping problem and let $\ep \geq 0$.
A random stopping time $\eta$ is \emph{$\ep$-optimal} if
\[ \gamma(\eta;X) \geq \sup_{\eta'}\gamma(\eta';X)-\ep, \]
where the supremum is taken over all random stopping times that have the same type as $\eta$.
Note that to qualify as $\ep$-optimal,
a random stopping time is compared only to other random stopping times of its own type.
The equivalence between the three types of random stopping times yields the following result.
\begin{corollary}\label{cor:1}
Let $X$ be a stopping problem and $\ep \geq 0$.
If $\eta$ is an $\ep$-optimal random stopping time,
and if $\eta$ and $\eta'$ are equivalent,
then $\eta'$ is an $\ep$-optimal random stopping time as well.
\end{corollary}

Using Theorem \ref{theorem:1},  Corollary \ref{cor:1} implies that the value of an optimal stopping problem is
independent of the notion of random stopping times that is being used.

\section{Random Stopping Times in Stopping Games}
\label{games}

In this section we study stopping games, that is, multiplayer stopping problems.
We show that Theorem \ref{th:2} and its consequences extend to this setup.

Given a finite set $J$, we denote by $\calP_*(J)$ the set of all nonempty subsets of $J$.

\begin{definition}
A  \emph{stopping game} is given by a finite set $J$ of players and an adapted and
integrable process $X=(X_n)_{n\in \dN\cup\{\infty\}}$ with values in $\dR^{J\times \calP_*(J)}$
\end{definition}

In a stopping game each player $j \in J$ chooses a stopping time $\sigma_j$.
As a function of the profile $\sigma=(\sigma_j)_{j\in J}$,
 player $j$'s payoff is given by
\begin{eqnarray*} \gamma^j(\sigma;X) &=&
\E[X^{j,J(\omega)}_{\sigma_*(\omega)}(\omega)],
\end{eqnarray*}
where $\sigma_*:=\min_{j'\in J}\sigma_{j'}$ is the time at which the game terminates
and $J(\omega):=\{j'\in J: \sigma_{j'}(\omega)=\sigma_*(\omega)\}$
is the set of players who chose to stop at that time.
From here on, we focus on the case of two players.
All the results extend to more-than-two-player games, with obvious changes.

Given a profile $\eta = (\eta_j)_{j \in J}$ of random stopping times, the \emph{detailed distribution} of $\eta$ is defined following Section \ref{model}.
To avoid duplication, we omit the details, which are standard.
For illustration, we explain how to adjust Section \ref{sec:randomized}.

Let  $\rho=(\rho_1,\rho_2) = (\rho_{1,n}, \rho_{2,n})_{n \in \N \cup \{\infty\}}$ be a pair of randomized stopping times,
not necessarily of the same type.
Set $(\ti\Omega,\ti\F,\ti P) = (I_1\times I_2 \times \Om, \calB_1\times \calB_2 \times \F, \lambda \otimes\lambda\otimes P)$,
where $(I_1,\calB_1,\lambda)$ and $(I_2,\calB_2,\lambda)$ are two copies of $(I,\calB,\lambda)$.
For every $(r_1,r_2,\omega) \in\ti\Om$ set
$\alpha(r_1,r_2,\omega) = \om$ and
\begin{equation}
\label{equ:tau:rho2}
\ti N^j_\rho(r_1,r_2,\omega) = \min\left\{n \in \N \colon \sum_{k=1}^n \rho_{i,k}(\omega) \geq r_j\right\},
\end{equation}
for each player $j=1,2$.
The detailed distribution $P_\rho$ of the pair  $\rho =(\rho_1,\rho_2)$ of randomized stopping times is given by
\[ \prob_\rho(A \times \{n_1\}\times \{n_2\}) = \ti P_{\ti N_\rho}(A \times \{n_1\}\times \{n_2\})
= \E_P[\One_A \rho_{1,n_1}\rho_{2,n_2}], \ \ \ \forall A \in \calF, \forall n_1,n_2 \in \dN \cup \{\infty\} .\]

This  definition formalizes  the implicit assumption that
the randomizations done by the players are independent.
Note that the detailed distribution $\prob_{(\rho_1,\rho_2)}$ is a probability distribution over $\Omega\times (\dN_1\cup\{\infty\})\times (\dN_2\cup\{\infty\})$,
whose marginals over $\Omega\times (\dN_1\cup\{\infty\})$ and $\Omega\times (\dN_2\cup\{\infty\})$ coincide with the detailed distributions of $\rho_1$ and $\rho_2$ respectively.

We now introduce another notion  of equivalence between random stopping times, motivated by stopping games.

\begin{definition}
Two random stopping times $\eta_1$ and $\eta'_1$ of Player 1, not necessarily of the same type, are \emph{game equivalent}
if for every random stopping time $\eta_2$ of player 2, one has
\[ \prob_{\eta_1,\eta_2} = \prob_{\eta'_1,\eta_2}. \]
\end{definition}

The expected payoff that a profile of random strategies $\eta = (\eta_j)_{j \in J}$ induces is defined analogously to the definition
in Section \ref{games}.
We here provide the definition only for pairs of randomized stopping times $(\rho_1,\rho_2)$.
\[
\gamma^j(\rho_1,\rho_2;X) := \E_P[X^{j,J(\omega)}_{\min\{\widetilde N^1_\rho,\widetilde N^2_\rho\}}], \]
where $J(\omega) = \{j \colon \widetilde N^j_\rho = \min\{\widetilde N^1_\rho,\widetilde N^2_\rho\}\}$ is the set of players who
stop at the stopping time $\min\{\widetilde N^1_\rho,\widetilde N^2_\rho\}$.

The following result is the analog of Theorem \ref{th:2} to stopping games. We omit the proof.
\begin{theorem}
\label{th:5.1}
Two random stopping times $\eta_1$ and $\eta'_1$ of Player 1 are game equivalent if, and only if,
for every stopping game $X$ and every random stopping time $\eta_2$ of Player 2,
\[ \gamma^j(\eta_1,\eta_2;X) = \gamma^j(\eta'_1,\eta_2;X), \ \ \ \forall j \in \{1,2\}.\]
\end{theorem}

The main result of this section is the following.
\begin{theorem}
Two random stopping times  of Player 1 are game equivalent if and only if they are equivalent.
\end{theorem}

\begin{proof}
The direct implication is straightforward. Indeed, let $\eta_1$ and $\eta'_1$ be two game-equivalent random stopping times of player 1,
and let $\eta_2$
be a random stopping time of player 2. By assumption, $\prob_{(\eta_1,\eta_2)}=\prob_{(\eta_1',\eta_2)}$.
Taking marginals, it follows that the detailed distributions $\prob_{\eta_1}$ and $\prob_{\eta'_1}$ coincide,
so that $\eta_1$ and $\eta'_1$ are equivalent.

We now turn to the reverse implication. We let $\eta_1$ and $\eta'_1$ be two equivalent random stopping times of player 1,
let $\eta_2$ be a random stopping time of player 2, and let $X$ be a two-player stopping game.
Building on Section \ref{model}, consider a filtered probability space $(\ti\Omega,\ti\calF, (\ti \calF_n)_{n \in \N},\ti P)$ derived from
the given filtered probability space $(\Omega,\calF, ( \calF_n)_{n \in \N} P)$, on which the two pairs $(\eta_1,\eta_2)$ and $(\eta'_1,\eta_2)$ can be presented as randomized integer-valued random variables.
In particular, the stopping time $\ti N_{\eta_2}$ is defined by
(\ref{equ:tau:rho}), (\ref{equ:tau:beta}), or (\ref{equ:tau:mu}).

Define an adapted real-valued process $\ti X=(\ti X_n)_{n \in \dN\cup \{\infty\}}$ by setting (i) $\ti X_n= X_n^{1,\{1\}}$ on the event $\ti N_{\eta_2} >n$,
(ii)  $\ti X_n= X_n^{1,\{1,2\}}$ on the event $\ti N_{\eta_2} =n$ and (iii) $\ti X_n= X_{\ti N_{\eta_2}}^{2,\{2\}}$. Intuitively, $\ti X$ is the optimal stopping problem faced by player 1 when player 2 is using the random stopping time $\eta_2$.

By construction, for every random stopping time $\bar \eta_1$ of Player 1 we have
\begin{equation}\label{equ:15}\gamma(\bar \eta_1;\ti X)=\gamma^1(\eta_1,\eta_2;X)
\end{equation}
Since $\eta_1$ and $\eta'_1$ are equivalent, one has $\gamma(\eta_1;\ti X)=\gamma(\eta'_1;\ti X)$, so that
$\gamma^1(\eta_1,\eta_2;X)=\gamma^1(\eta'_1,\eta_2;X)$. Since $\eta_2$ and $X$ are arbitrary, and by Theorem \ref{th:5.1}, this implies that
$\eta_1$ and $\eta'_1$ are game equivalent.
\end{proof}

\bigskip

We conclude by listing a few direct consequences of the latter results. We start with zero-sum games.
A two-player stopping game is \emph{zero-sum} if $X^{1,C}+X^{2,C} =  0$, for each nonempty subset $C$ of $\{1,2\}$.
In two-player zero-sum stopping games, given $\ep \geq 0$,
a random stopping time $\eta^*_1$ is \emph{$\ep$-optimal} for Player 1 if
\[ \inf_{\eta_2}\gamma^1(\eta^*_1,\eta_2;X) \geq \sup_{\eta_1}\inf_{\eta_2}\gamma^1(\eta_1,\eta_2;X)-\ep, \]
where the inf and sup are over all random stopping times that have the same type as $\eta^*_1$.
$\ep$-optimal stopping times for Player 2 are defined analogously.
The equivalence between the three types of random stopping times delivers the following result.

\begin{theorem}
\label{th:value1}
If $\eta^*_i$ is an $\ep$-optimal random stopping time for player $i$,
then
\[ \inf_{\eta_2}\gamma^1(\eta^*_1,\eta_2;X) \geq \sup_{\eta_1}\inf_{\eta_2}\gamma^1(\eta_1,\eta_2;X)-\ep, \]
where the infimum and  supremum are taken over \emph{all} random stopping times (of all three types).
\end{theorem}

The equivalence also shows that any random stopping time that is equivalent to some $\ep$-optimal stopping time
is also $\ep$-optimal.
\begin{theorem}
\label{th:value2}
If $\eta_i$ is an $\ep$-optimal stopping time for player $i$,
and if $\eta'_i$ is equivalent to $\eta_i$,
then $\eta'_i$ is an $\ep$-optimal random stopping time for player $i$.
\end{theorem}

Given a class (randomized, behavior, or mixed) of random stopping times, we say that the game has a \emph{value} in that class if $\sup_{\eta_1}\inf_{\eta_2}\gamma^1(\eta_1,\eta_2;X) = \inf_{\eta_2}\sup_{\eta_1}\gamma^1(\rho_1,\rho_2;X)$ where the supremum and infimum are taken
over $\eta_1$ and $\eta_2$ in that class. The common value of $\sup\inf=\inf\sup$ is called the value of the game.

By Theorem \ref{th:value1},  the existence of the value does not hinge on which class of random stopping times is being considered,
and the value of the game remains the same.
Rosenberg, Solan, and Vieille (2001) proved that two-player zero-sum stopping games have a value in behavior stopping times.
We thus obtain the following result.
\begin{theorem}
Every two-player zero-sum stopping game has a value in randomized stopping times and a value in mixed stopping times.
\end{theorem}

We now turn to two-player nonzero-sum games.
Given $\ep \geq 0$,
a pair $(\rho^*_1,\rho^*_2)$ of randomized stopping times is an \emph{$\ep$-equilibrium} if for every
other pair of randomized stopping times $(\rho_1,\rho_2)$ we have
\[ \gamma^1(\rho^*_1,\rho^*_2;X) \geq \gamma^1(\rho_1,\rho^*_2;X) - \ep,
\ \ \ \ \
\gamma^1(\rho^*_1,\rho^*_2;X) \geq \gamma^1(\rho^*_1,\rho_2;X) - \ep. \]
$\ep$-equilibria in behavior stopping times and in mixed stopping times are defined analogously.
Analogously to Theorems \ref{th:value1} and \ref{th:value2} we have the following two results.

\begin{theorem}
If $(\rho^*_1,\rho^*_2)$ is an $\ep$-equilibrium in randomized stopping times, then
for every pair of random stopping times $(\eta_1,\eta_2)$ we have
\[ \gamma^1(\rho^*_1,\rho^*_2) \geq \gamma^1(\eta_1,\rho^*_2;X) - \ep,
\ \ \ \ \
\gamma^1(\rho^*_1,\rho^*_2;X) \geq \gamma^1(\rho^*_1,\eta_2;X) - \ep. \]
\end{theorem}
An analog result holds for $\ep$-equilibria in behavior stopping times and for $\ep$-equilibria in mixed stopping times.

\begin{theorem}
\label{th:13}
If $(\rho^*_1,\rho^*_2)$ is an $\ep$-equilibrium in randomized stopping times,
and if for each $j \in \{1,2\}$ the behavior stopping time $\beta^*_j$ is equivalent to $\rho^*_j$,
then $(\beta^*_1,\beta^*_2)$ is an $\ep$-equilibrium in behavior stopping times.
\end{theorem}
There are five analog theorems to Theorem \ref{th:13},
depending on the type of strategies in the given $\ep$-equilibrium
and on the type of equivalent strategies.

Shmaya and Solan (2004) proved that every two-player stopping game has an $\ep$-equilibrium in behavior stopping times.
The equivalence of the three types of random stopping times gives us the following.
\begin{theorem}
Every two-player nonzero-sum stopping game has an equilibrium in randomized stopping times
and an equilibrium in  mixed stopping times.
\end{theorem}

\end{document}